\documentclass{amsart}

\usepackage[utf8]{inputenc}		

\usepackage{mathtools}			

\usepackage{amsfonts}			
\usepackage{amssymb}			
\usepackage{amsthm}				
\usepackage[giveninits=true,url=false,isbn=false,maxbibnames=99,sortcites,sorting=nyt]{biblatex}	
\usepackage{booktabs}			
\usepackage{enumerate}
\usepackage[a4paper]{geometry}	
\usepackage{lmodern}			
\usepackage{mathrsfs}			
\usepackage{microtype}			
\usepackage{multicol}
\usepackage{pdflscape}
\usepackage{xcolor}				

\usepackage[colorlinks=true,linkcolor=red!50!black,citecolor=green!50!black,urlcolor=blue!50!black]{hyperref} 
\usepackage[capitalise]{cleveref}	

\renewbibmacro{in:}{}   
\DeclareFieldFormat[article]{title}{#1} 

\newtheorem{thm}{Theorem}[section]
\newtheorem{corollary}[thm]{Corollary}
\newtheorem{lemma}[thm]{Lemma}

\newtheorem{introthm}{Theorem}

\newtheorem*{thm*}{Theorem}
\newcounter{claim}[thm] 

\renewcommand{\theintrothm}{\Alph{introthm}}

\theoremstyle{definition}

\newtheorem{remark}[thm]{Remark}

\crefname{introthm}{Theorem}{Theorems}
\newtheorem{example}[thm]{Example}

\renewcommand{\epsilon}{\varepsilon}

\newcommand{\ce}[2]{C_{#1}(#2)}

\newcommand{\nin}{\notin}
\newcommand{\norm}{\trianglelefteq}

\DeclareMathOperator{\Aut}{Aut}

\DeclareMathOperator{\cs}{cs}

\DeclareMathOperator{\soc}{soc}

\DeclareMathOperator{\Syl}{Syl}

\DeclareMathOperator{\CO}{CO}
\DeclareMathOperator{\GL}{GL}

\DeclareMathOperator{\PGammaU}{P\Gamma U}

\DeclareMathOperator{\PGL}{PGL}
\DeclareMathOperator{\PGO}{PGO}
\DeclareMathOperator{\PGSp}{PGSp}
\DeclareMathOperator{\PGU}{PGU}
\DeclareMathOperator{\POmega}{P\Omega}
\DeclareMathOperator{\PSL}{PSL}
\DeclareMathOperator{\PSp}{PSp}

\DeclareMathOperator{\PSU}{PSU}
\DeclareMathOperator{\Sp}{Sp}

\DeclareMathOperator{\SO}{SO}

\DeclareMathOperator{\D}{D}
\DeclareMathOperator{\G}{G}
\DeclareMathOperator{\F}{F}

\DeclarePairedDelimiter{\abs}{\lvert}{\rvert}

\DeclarePairedDelimiter{\gen}{\langle}{\rangle}

\title{Groups with conjugacy classes of coprime sizes}

\author{R.D. Camina}
\address{Fitzwilliam College, Cambridge, CB3 0DG, United Kingdom.}
\email{rdc26@cam.ac.uk}

\author{A. Maróti}
\address{Hun-Ren Alfr\'ed R\'enyi Institute of Mathematics, Hungarian Academy of Sciences, Re\'altanoda utca 13-15, H-1053, Budapest, Hungary.}
\email{maroti.attila@renyi.hu}

\author{E. Pacifici}
\address{Dipartimento di Matematica e Informatica ``U. Dini" (DiMaI), Universit\`a degli Studi di Firenze, Viale Morgagni 67/a, 50134 Firenze, Italy.}
\email{emanuele.pacifici@unifi.it}

\author{C. Parker}
\address{School of Mathematics, University of Birmingham, Birmingham, B15 2TT, United Kingdom}
\email{c.w.parker@bham.ac.uk}

\author{K. Rekvényi}
\address{Department of Mathematics, University of Manchester, Manchester, M13 9PL, United Kingdom. Also affiliated with: Heilbronn Institute for Mathematical Research, Bristol, BS8 1UG, United Kingdom.}
\email{kamilla.rekvenyi@manchester.ac.uk}

\author{J. Saunders}
\address{School of Mathematics, University of Bristol, Bristol, BS8 1UG, United Kingdom. Also affiliated with: Heilbronn Institute for Mathematical Research, Bristol, BS8 1UG, United Kingdom.}
\email{J.P.Saunders@bristol.ac.uk}

\author{V. Sotomayor}
\address{Departamento de Álgebra, Universidad de Granada, Avda. de Fuentenueva s/n, 18071 Granada, Spain.}
\email{vsotomayor@ugr.es}

\author{G. Tracey}
\address{Mathematics Institute, University of Warwick, Coventry CV4 7AL, United Kingdom}
\email{Gareth.Tracey@warwick.ac.uk}

\author{M. van Beek}
\address{Department of Mathematics, University of Manchester, Manchester, M13 9PL, United Kingdom. Also affiliated with: Heilbronn Institute for Mathematical Research, Bristol, BS8 1UG, United Kingdom.}
\email{martin.vanbeek@manchester.ac.uk}

\date{\today}

\begin{document}

\begin{abstract}    
    Suppose that \(x\), \(y\) are elements of a finite group \(G\) lying in conjugacy classes of coprime sizes. We prove that \(\gen{x^G} \cap \gen{y^G}\) is an abelian normal subgroup of \(G\) and, as a consequence, that if \(x\) and \(y\) are \(\pi\)-regular elements for some set of primes \(\pi\), then \(x^G y^G\) is a \(\pi\)-regular conjugacy class in \(G\). The latter statement was previously known for \(\pi\)-separable groups \(G\) and this generalisation permits us to extend several results concerning the \emph{common divisor graph} on $p$-regular conjugacy classes, for some prime $p$.
    \medskip
    
    
    \smallskip
    
\end{abstract}

\maketitle


\section{Introduction} 

Let $G$ be a finite group, $H \le G$ and $x \in G$. Then $x^H=\{x^h\mid h\in H\}$ is the set of $H$-conjugates of $x$ and $x^G$ is the conjugacy class of $x$.  Analysis of the set \(\cs(G)\) of the conjugacy class sizes of $G$ is a well-established research theme as it can reveal structural information about the group $G$. In this article we consider the case where \(G\) has a pair of non-trivial conjugacy classes of coprime sizes. Our original motivation is outlined below, but a key step in reaching our goal was the following theorem, which we anticipate is of independent interest.

\begin{introthm}   \label{AlmostSimple}
    Suppose that \(G\) is a finite almost simple group. Then \(G\) does not have two non-trivial conjugacy classes of coprime sizes.
\end{introthm}

As might be expected, our proof of \cref{AlmostSimple} (and thus all of our main results), relies on the classification of finite simple groups.

If a group \(G\) has conjugacy classes \(x^G\) and \(y^G\) of coprime sizes then \(G = C_G(x) C_G(y)\) (see \cref{lem:fund lemma}), so to eliminate an almost simple group $G$ it is sufficient to show that \(G\) is not a product of centralisers. In the simple case this is Szép's conjecture \cite{Szep} which states that no non-abelian finite simple group \(G\) has such a factorisation. This was proved by Arad and Fisman in \cite[Theorem 1]{Szep1}. In our case, we employ a variant of this result due to Gill, Giudici and Spiga in \cite{Szep2} which determines all triples \((G,x,y)\) with \(G\) almost simple such that \(G = C_G(x) C_G(y)\) (see \cref{GillGiudiciSpigaTable} below) provided the socle of \(G\) is not \(\POmega_n^+(q)\) with $n \ge 8$ even. To prove \cref{AlmostSimple}, we show that the candidates for \(x\) and \(y\) in \cref{GillGiudiciSpigaTable} do not have coprime class sizes. In the remaining case where the socle of $G$ is \(\POmega_n^+(q)\), a close analysis of the possibilities for \(x\) and \(y\) again yields no examples.

Using \cref{AlmostSimple}, we obtain our first main result which gives structural information about any finite group containing non-trivial conjugacy classes of coprime sizes.

\begin{introthm}  \label{ThmBcor2}
   Suppose that $G$ is a finite group and $x$, $y \in G$ are in conjugacy classes of coprime sizes. Then the normal subgroup $\gen{x^G}\cap \gen{y^G}$ is abelian.  Furthermore, we can write 
   $$\gen{x^G}\cap \gen{y^G}= (\gen{x^G}\cap Z(\gen{y^G}))(Z(\gen{x^G}) \cap \gen{y^G}).$$
\end{introthm}

Recall that a component of a finite group \(G\) is a quasisimple subnormal subgroup. As a consequence of \cref{ThmBcor2} we also obtain the following theorem.

\begin{introthm}    \label{thmB}
    Suppose that \(G\) is a finite group and \(x\), \(y \in G\) lie in conjugacy classes of coprime sizes. Assume that $K$ is a component of $G$ or $K=O_r(G)$ for some prime $r$. Then \(\gen{K^G}\) is centralised by \(x\) or by \(y\). In particular, if \(G\) is almost simple then \(x = 1\) or \(y = 1\).
\end{introthm}

Our original motivation relates to the \emph{common divisor graph} $\Gamma(G)$, which is a useful tool devised to explore the arithmetical structure of \(\cs(G)\). The vertex set of $\Gamma(G)$ is $\cs(G)\setminus\{1\}$, and distinct vertices are adjacent if and only if they are not coprime integers. The definition of this graph is inspired by a closely related graph introduced in \cite{BHM} by Bertram, Herzog and Mann: their graph has as vertex set the non-central conjugacy classes of $G$, with two of them adjacent if and only if their sizes are not coprime. It is worth pointing out that this last graph and $\Gamma(G)$ have the same number of connected components. Furthermore, the connected components have the same diameter (except when $G$ has at least two non-central conjugacy classes of the same size and this size is coprime with any other class size; see Section~8 of \cite{Lewis}).
As an immediate application of the aforementioned result of Arad and Fisman, it is noted in \cite{BHM} that $\Gamma(G)$ is complete for simple groups $G$. Now, our \cref{AlmostSimple} generalises it by showing that \(\Gamma(G)\) is complete if \(G\) is an almost simple group. 

Recall that, for a set of primes \(\pi\), we say an element is \emph{$\pi$-regular} if its order is not divisible by any prime in \(\pi\), and a conjugacy class is \emph{$\pi$-regular} if it consists of $\pi$-regular elements. If such an element is \(\{p\}\)-regular for some prime $p$ then we write \emph{\(p\)-regular} for simplicity. Consider the induced subgraph \(\Gamma_{p}(G)\) of \(\Gamma(G)\) defined on the \(p\)-regular conjugacy classes of \(G\). This graph is also closely related to a graph that has been studied extensively over the years (mainly by Beltrán and Felipe, see \cite{BF1,BF2,BF4,BF5}) with a convenient summary found in \cite{BFsurvey}. It is important to remark that the graphs \(\Gamma(G)\) and \(\Gamma_{p}(G)\) share many properties, including upper bounds on both the number of connected components and their diameters. Nevertheless, these properties are not immediate consequences of the fact that \(\Gamma_p(G)\) is a subgraph of \(\Gamma(G)\). We also note here that if $p$ does not divide \(\abs{G}\) then \(\Gamma_{p}(G)\) is the same as \(\Gamma(G)\), so results for \(\Gamma_{p}(G)\) are, in fact, generalisations of the corresponding results for \(\Gamma(G)\).

The local structure of \(G\) has been used to determine many properties of \(\Gamma_p(G)\), but much of the corresponding analysis has been restricted to the case where \(G\) is \(p\)-soluble. A core lemma in the proofs of many of these results is \cite[Lemma 1 (2)]{LZ} which says that if \(G\) is \(p\)-soluble then the product of two $p$-regular classes of coprime sizes is itself a $p$-regular class. In many cases, this is the only use of the \(p\)-solubility hypothesis. The initial motivation of this article was to remove the dependency on \(p\)-solubility from these results by generalising \cite[Lemma 1 (2)]{LZ} so that they hold for all finite groups. This is achieved in our final main result which follows from \cref{ThmBcor2} and settles a conjecture stated in \cite{S}. In fact, we demonstrate it in full generality for $\pi$-regular conjugacy classes.

\begin{introthm}    \label{thmA}
    Let $G$ be a finite group and $\pi$ be a set of primes. If $x$ and $y$ are \(\pi\)-regular elements of $G$ whose conjugacy classes have coprime sizes, then $x^G y^G$ is a \(\pi\)-regular conjugacy class.
\end{introthm}

Finally, as already mentioned in \cite[Remark 2.2]{S}, \cref{thmA} implies that several results in the literature concerning graphs defined on $p$-regular conjugacy classes in $p$-soluble groups remain valid even when removing the assumption of $p$-solubility. In \cref{sec:corollaries} we will present a detailed discussion of these generalisations.

\section*{Acknowledgements}

This paper is based upon work supported by the Heilbronn Institute Focused Research Workshop ``\href{https://heilbronn.ac.uk/2025/03/12/focused-research-finite-groups-conjugacy-classes-and-related-problems/}{Finite groups, conjugacy classes and related problems}'', during a visit of all the authors, except the seventh, to the School of Mathematics at the University of Bristol (UK) in June 2025 organised by the fifth and sixth authors. The second author was partially supported by the National Research, Development and Innovation Office (NKFIH) Grant No.~K138596 and Grant No.~K138828. The research of the third author was partially supported by INdAM-GNSAGA, and by the European Union-Next Generation EU, Missione 4 Componente 1, CUP B53D23009410006, PRIN 2022 2022PSTWLB - Group Theory and Applications. The fifth, sixth and ninth authors were supported by the Additional Funding Programme for Mathematical Sciences, delivered by EPSRC (EP/V521917/1) and the Heilbronn Institute for Mathematical Research. 

\section{Preliminaries}

All groups considered hereafter are assumed to be finite. The greatest common divisor of two positive integers $a$ and $b$ is denoted $(a,b)$. For a set of primes $\pi$, a $\pi$-number is a natural number whose prime divisors all lie in $\pi$. We denote by $\pi'$ the set of primes not in $\pi$.  
The largest $\pi$-number that divides the positive integer $n$ is denoted by $n_{\pi}$. For an integer \(n\) we define \(\pi(n)\) to be the set of prime divisors of \(n\) and for a finite group \(G\) we use \(\pi(G)\) to denote \(\pi(\abs{G})\).

The following well-known lemma will be used throughout the article, often without reference. We include its proof for completeness. 

\begin{lemma}   \label{CollectionLemma}
    Let $G$ be a finite group, $N$ a normal subgroup of $G$, and $x \in G$.  Then $\abs{x^N} \abs{(xN)^{G/N}}$ divides $\abs{x^G}$.
\end{lemma}

\begin{proof}
We have that $\abs{G : NC_G(x)}= (\abs{G} \abs{C_N(x)})/(\abs{N} \abs{C_G(x)}) = \abs{x^G}/\abs{x^N}$ is an integer as $NC_G(x)$ is a subgroup of $G$. Since $\ce{G}{x}N/N \leq \ce{G/N}{xN}$, we have that $\abs{(xN)^{G/N}}$ divides $\abs{G : N\ce{G}{x}}$.
\end{proof}

We also highlight the following key fact which is essentially \cite[Lemma 1]{BHM}.

\begin{lemma}   \label{lem:fund lemma}
    Suppose that \(G\) is a finite group and $x$, $y \in G$ are such that $(\abs{x^G}, \abs{y^G}) = 1$. Then, for any \(r \in \pi(G)\), at least one of \(C_G(x)\) and \(C_G(y)\) contains a Sylow \(r\)-subgroup of \(G\). In particular, \(G = C_G(x) C_G(y)\) and, further, \((xy)^G = x^G y^G\).
\end{lemma}

\begin{proof}
    The first claim is clear from the hypothesis, and that \(G = C_G(x) C_G(y)\) follows from the fact that these centralisers have coprime index and for \(H\), \(K \leq G\) we have \(\abs{HK} = \abs{H} \abs{K}/\abs{H \cap K}\). Now, let \(g\), \(h \in G\) and write \(gh^{-1} = ab\) for some \(a \in C_G(x)\), \(b \in C_G(y)\). Then
    \[x^g y^h = (x^{gh^{-1}} y)^h = (x^{ab} y)^h = (x^b y^b)^h = (xy)^{bh}\]
    as required.
\end{proof}

\begin{lemma}   \label{lem:not sol sbgrp} 
    Let \(G\) be a finite group, let $\pi \subseteq \pi(G)$ be a set of primes and let \(X \le G\) be a \(\pi\)-separable subgroup of $G$. Suppose further that \(x\), \(y \in G\) are \(\pi\)-regular elements of $G$ with \((\abs{x^G}, \abs{y^G}) = 1\). If \(x^G \cap X \neq \emptyset \neq y^G \cap X\) then \(xy\) is also \(\pi\)-regular.
\end{lemma}
\begin{proof} 
    Set $L=x^G$ and $K=y^G$. Assume that $a\in L \cap X$, $b\in K \cap X$. Then by \cite[Theorem 6.3.6]{Gorenstein} there exists a Hall $\pi'$-subgroup $H$ of $X$ such that $a^X\cap H \ne \emptyset \ne b^X \cap H$. Letting $u\in a^X\cap H \subseteq L\cap H$ and $w \in b^X \cap H\subseteq K\cap H$, since $H$ consists of $\pi$-regular elements, we have that $uw$ is $\pi$-regular. As $KL$ is a conjugacy class by \cref{lem:fund lemma}, and contains $uw$, we conclude that $xy$ is $\pi$-regular. This proves the lemma.
\end{proof}

\begin{lemma}   \label{lem:is normal} Let \(G\) be a finite group with \(x\), \(y \in G\) such that \(G = C_G(x) C_G(y)\). Suppose that $C_G(y) \le H \le G$ and $x^G \cap H \neq \emptyset$. Then $\gen{x^G} \leq H$.

\end{lemma}

\begin{proof} 

    We may assume that $x \in H$. Then $x^G = x^{C_G(x)C_G(y)} = x^{C_G(y)} \subset x^H \subset H$, as required.
\end{proof}

The following lemma collects information we require from \cite[Theorem 1.1]{Szep2}.

\begin{lemma}   \label{GillGiudiciSpiga}
    Suppose that \(G\) is an almost simple group whose socle \(Y\) is not isomorphic to \(\POmega_n^+(q)\) for \(n \geq 8\). Assume that \(x\), \(y \in G\) are such that \(G = C_G(x) C_G(y)\). Then \((G,x,y)\) is recorded in \cref{GillGiudiciSpigaTable} (up to swapping \(x\) and \(y\)).
\end{lemma}

\begin{table} [h]
        \begin{tabular}{cccc}   \toprule
        \(\soc(G)\)         &   \(x\)                   &   \(y\)                   &   Remarks \\ \midrule
        \(\PSL_n(q)\)       &   graph aut.   &   order dividing \(q-1\)  &   \(n \geq 4\) even  \\
                            &   \(C_{\PGL_n(q)}(x) \cong \PGSp_n(q)\)   &   \((n-1)\)-dim. eigenspace   &      \\[10pt]
        \(\PSU_n(q)\)       &   \(C_Y(x) \cong \PSp_n(q)\)     &   order dividing \(q+1\)  &   \(n \geq 4\) even  \\
                            &   \(x \in \PGammaU_n(q) \setminus \PGU_n(q)\)     &   \((n-1)\)-dim. eigenspace   &   \(\abs{G:Y}\) even  \\[10pt]
        \(\PSp_n(q)\)       &    \(x \in \PGSp_n(q) \setminus \PSp_n(q)\)   &   order \(r\), transvection   &   \(q\) odd, \(4 \mid n\)    \\
                            &   \(C_Y(x) \cong \PSp_{n/2}(q^2).2\)   &   &   \(\abs{G:Y}\) even   \\[10pt]
        \(\POmega_n^-(q)\)  &   graph aut.  &   order dividing \(q+1\)  &   \(n/2\) odd \\
                            &   \(C_Y(x) \cong \Omega_{n-1}(q)\) if \(q\) odd   &   no eigenvalue in \(\mathbb{F}_q\)   &      \\
                            &   \(C_Y(x) \cong \Sp_{n-2}(q)\) if \(q\) even &   &   \\[10pt]
        \(\POmega_n(q)\)    &   \(C_Y(x) \cong \POmega_{n-1}^-(q).2\)  &   order \(r\), unipotent  &   \(n \equiv 1 \pmod 4\)  \\
                            &   \(x \in \SO_n(q) \setminus \Omega_n(q)\)    &   \(C_Y(y) \cong E_q^{m(m-1)/2+m}:\Sp_m(q)\)   &   \(\SO_n(q) \leq G\) \\ \bottomrule
        \end{tabular}\vspace{3mm}
        \caption{An excerpt from \cite[Table 1]{Szep2} containing the cases where \(G = C_G(x) C_G(y)\). In all cases \(x\) has order 2.}
        \label{GillGiudiciSpigaTable}
\end{table}

We use the below lemma to avoid some difficulties with small cases in \cref{AlmostSimple}.

\begin{lemma}   \label{computer}
    Suppose that \(G\) is an almost simple group with socle \(\Sp_4(2)'\), \(\G_2(2)'\), \({}^2\!\F_4(2)'\), \({}^2\!\G_2(3)'\) or \(\Omega_n^+(2)\) for \(n \in \{8, 10, 12\}\). Then \(G\) does not have a pair of non-trivial conjugacy classes of coprime sizes.
\end{lemma}

\begin{proof}
    This can be verified directly by computing the conjugacy class sizes in {\sc GAP} or {\sc MAGMA}.
\end{proof}

\section{Proof of the main results}

In this section, we present the proofs of \cref{AlmostSimple,thmA,thmB,ThmBcor2}. We begin with \cref{AlmostSimple} as it is the key ingredient in the proofs of the other three.

\begingroup
\def\theintrothm{\ref{AlmostSimple}}
\begin{introthm}   
    Suppose \(G\) is a finite almost simple group. Then \(G\) does not have two non-trivial conjugacy classes of coprime sizes.
\end{introthm}
\addtocounter{introthm}{-1}
\endgroup  

\begin{proof}
    Suppose that the statement is false, so $G$ is an almost simple group not covered by \cref{computer} and we have non-trivial elements \(x\), \(y \in G\) such that \((\abs{x^G}, \abs{y^G}) = 1\). Then $G = C_G(x)C_G(y)$ by \cref{lem:fund lemma}. Since $C_G(z) \le C_G(z^n)$ for any $n \in \mathbb{N}$ and $z\in G$,  we may assume that $x$ and $y$ have prime order. By \cite[Theorem B]{Szep1}, $G$ is not a simple group. Let $Y=F^*(G)$ be the generalised Fitting subgroup of \(G\). By \cite[Theorem 1.1 and Table 1]{Szep2}, \(Y\) is a group of Lie type defined in some characteristic, \(r\) say. Let $R \in \Syl_r(G)$. By \cref{lem:fund lemma}, we may suppose that $y$ centralizes $R$. Then, because of \cref{computer}, we may apply \cite[Lemma D.25]{PPSS}  to obtain  $y \in Y$.  Thus $y \in C_Y(R\cap Y)$.  Since $R \cap Y = O_r(N_{Y}(R\cap Y)) = F^*(N_{Y}(R \cap Y)) $ by \cite[Corollary 3.1.4]{GLS3}, we have  $C_Y(R\cap Y) = C_Y(F^*(N_Y(R \cap Y)))=Z(R\cap Y)$, we have that $C_Y(R) \le Z(R\cap Y)$. Therefore \(y \in Z(R)\cap Y \le Z(R \cap Y)\). Furthermore, by \cref{CollectionLemma} we have that \(x \nin Y\). 
    
    If $Y \ncong \POmega_{2m}^{+}(q)$ with $m \geq 4$, then \cref{{GillGiudiciSpiga}}  provides a list of the possibilities for an almost simple group $H$ and non-trivial elements $a$, $b \in H$ satisfying $H=C_H(a)C_H(b)$. There are five families of groups to consider, listed in \cref{GillGiudiciSpigaTable}. We claim that in all cases the class sizes in \cref{GillGiudiciSpigaTable} are not coprime. 
    Since  \(x \in G \setminus Y\) has prime order, \cref{{GillGiudiciSpiga}} implies $x$ has order $2$ and we already know \(y \in Z(R)\le Z(R \cap Y)\). The only case which matches this requirement is line 3 of \cref{GillGiudiciSpigaTable} since in lines 1, 2 and 4 \(y\) is not an \(r\)-element and, in line 5, \(y\) does not lie in the centre of a Sylow \(r\)-subgroup. In the situation of line 3, $r$ is odd, $Y \cong \PSp_{2m}(q)$ with $m$ even and $y$ is the image of a transvection on the natural module $V$ for $\Sp_{2m}(q)$. Since $\Sp_{2m}(q)$ acts transitively on the $1$-dimensional subspaces of $V$, and since the number of these subspaces is even, we conclude that $C_G(y)$ has even index in $G$. Hence $C_G(x)$ contains a Sylow $2$-subgroup of $G$ by \cref{lem:fund lemma}. However, from line 3 of \cref{GillGiudiciSpigaTable} we see that \(C_Y(x) \cong \PSp_m(q^2).2\) and can directly compute that this has even index in \(Y\) and thus $C_G(x)$ has even index in \(G\), a contradiction. We conclude that \(Y \cong \POmega_{2m}^+(q)\) for $m\geq 4$ and \(y \in Z(R) \leq Z(R\cap Y)\). 
    
    Define $M= C_Y(y)$ and note that \(R \leq C_G(y)\). If $x$ has order $r$, then we may assume that $x\in R$. But then $x$ centralises \(y\) and we obtain $Y\le \gen{x^G}\le N_G(\gen{y})$ from \cref{lem:is normal}. Since $\gen{y}$ is not normal in $Y$, this is impossible. Hence $x$ does not have order $r$.

    Let \(\widehat H = \CO_{2m}^+(q)\) be the conformal orthogonal group of $+$ type with natural module \(V\) and \(\widehat Y = \Omega_{2m}^+(q) \leq \widehat H\). Notice that taking a preimage $\widehat y$ of $y$ in $\widehat H$ of order $r$ we obtain a decomposition $V=[V,\widehat y] + C_V(\widehat y)$ with $W=[V,\widehat y]$ a totally singular $2$-space (see, for example \cite[Lemma 14.33 vi)]{ParkerRowleySymplecticAmalgams}). In particular, $\widehat M$ stabilises a totally singular $2$-space of $V$.     
    
    In what follows, note that we use \cref{computer} to sidestep the fact that $2^6-1$ has no primitive prime divisor \cite{Zsigmondy}. Let \(M_2 \leq \widehat Y\) be such that \(M_2 \cong \Omega_{2m-2}^-(q)\) and \(M_2\) fixes pointwise a 2-dimensional subspace \(V_2\) of \(-\) type. Then \(M_2\) contains a cyclic subgroup of order \(q^{m-1}+1\)  acting irreducibly on $V_2^\perp$ (see \cite[Satz 3]{Huppert}). Let \(\alpha\) be a primitive prime divisor of \(q^{2m-2}-1\). Similarly, we have a subgroup \(M_4 \cong \Omega_{2m-4}^-(q)\) of \(\widehat Y\) which fixes pointwise a 4-dimensional subspace \(V_4\) of \(-\) type and contains a cyclic subgroup of order \(q^{m-2}+1\). Let \(\beta\) be a primitive prime divisor of \(q^{2m-4}-1\). Let \(\widehat A \in \Syl_{\alpha}(M_2)\) and \(\widehat B \in \Syl_{\beta}(M_4)\). Then \(\widehat A\) acts irreducibly on \(V_2^{\perp}\) and centralises \(V_2\), and similarly \(\widehat B\) acts irreducibly on \(V_4^{\perp}\) and centralises \(V_4\). In particular, neither \(\widehat A\) nor \(\widehat B\) stabilise a totally singular \(2\)-space. Hence \(M\) does not contain a Sylow \(\alpha\)- or \(\beta\)-subgroup of \(G\). Moreover, $\widehat Y$ has a subgroup isomorphic to $\GL_m(q)$ which, in turn, contains a cyclic subgroup of order $q^m - 1$. Let $\gamma$ be a primitive prime divisor of $q^m - 1$. Then we may choose an element $\widehat w$ of order $\gamma$ which acts with two irreducible composition factors each of dimension $m$ on $V$. In particular, $w$ is not conjugate into $M$. Hence $M$ does not contain a Sylow $\gamma$-subgroup of $Y$. Therefore  we may assume $\gen{A, B, w} \leq C_Y(x)$. 

    Assume for a moment that $2m = 8$ and that $x$ has order $3$. Then, as $x$ does not have order $r$, $r\ne 3$. If $x$ acts as a graph automorphism then $C_Y(x) \cong \G_2(q)$ by \cite[Table 4.7.3A]{GLS3} which does not have order divisible by $\gamma = \beta$. If $x$ acts as a field automorphism, then $x$ normalises a long root subgroup as does $M$ and this contradicts \cref{lem:is normal}. If $x$ acts as a graph-field automorphism then by \cite[Proposition 4.9.1 (a) and (e)]{GLS3} $C_Y(x) \cong {}^3\!\D_4(q^{1/3})$ and has order not divisible by $\gamma$. Thus when \(2m = 8\), $x$ does not have order $3$.
    
    Let $m\geq 4 $ and suppose that $x$ acts on $Y$ as either a field automorphism or a graph-field automorphism. Since \(x\) does not have order \(r\), conjugacy of such automorphisms is the subject of \cite[Proposition 4.9.1]{GLS3}. In particular, we see that $x$ normalises a conjugate of $Z(R \cap Y)$ and this leads to a contradiction via \cref{lem:is normal}.

    Thus $x$ acts as an inner-diagonal or graph automorphism of $Y$. Then, as $x$ does not have order $3$ when $2m=8$, $x$ is an involution. Since $x$ does not have order $r$, we have $r$ is odd. So, conjugating by the triality automorphism if necessary, we may find some \(\widehat G \leq \widehat H\) such that \(\widehat G\) projects onto \(G\) and $\widehat Y = \Omega_{2m}^+(q)$. 
    
    Let \(\widehat x \in \widehat H\) be an element in the preimage of \(x\) chosen to have minimal order. Thus \(\widehat x\) is a \(2\)-element. As \(\alpha\), \(\beta\) and \(\gamma\) are coprime to \(\abs{Z(\widehat H)}\), \(\widehat x\) commutes with \(\widehat A\), \(\widehat B\) and \(\widehat C\). Furthermore, the choice of \(\widehat A\), \(\widehat B\) and \(\widehat C\) implies that \(C_{\widehat Y}(\widehat x)\) acts irreducibly on \(V\). If \(\widehat x\) is an involution, then \(\widehat x \in Z(\widehat H)\), a contradiction. Hence \(\widehat x\) has order at least 4.
    
    Suppose that $\widehat x$ leaves a totally singular 2-space invariant. Then $N_{\gen{\widehat Y, \widehat x}}(W) \cap \widehat x^G \neq \emptyset$, and so \cref{lem:is normal} yields  that $\widehat Y \le \gen{\widehat x^{\widehat {Y}}}$ leaves $W=[V,\widehat y]$ invariant, a contradiction. Hence $\widehat x$ does not leave a totally singular $2$-space invariant. 
     
    We now consider all of the possibilities for \(x\). Since \(x\) is an involution and \(r\) is odd, we may use \cite[Table B.10]{BG} (and the additional description provided in \cite[\nopp 3.5.2]{BG}) for a description of the possibilities for $\widehat x$. Since \(\widehat x\) has order at least 4, we know that it is not one of the involutions listed in \cite[3.5.2.1 to 3.5.2.6, 3.5.2.14]{BG}.
    
    From \cite[\nopp 3.5.2.7, 3.5.2.11 and 3.5.2.15]{BG} we see that $\abs{C_{Y}(x)}$ is not divisible by both $\alpha$ and $\beta$. So these cases cannot be candidates for $\widehat x$. Finally, the description provided in \cite[\nopp 3.5.2.10, 3.5.2.12]{BG} shows that these possibilities for \(\widehat x\) leave a totally singular $2$-space invariant and so these also cannot be $\widehat x$. The involutions listed in \cite[\nopp 3.5.2.16]{BG} are conjugate into $\PGO_{2m}^+(q)$ by the triality automorphism and so these cases can be excluded. This completes the investigation of all possibilities for \(x\) and \(y\), thus \(Y\) cannot be isomorphic to \(\POmega_{2m}^+(q)\).
\end{proof}

We also have the following as a corollary.

\begin{corollary}   \label{thmBcor}
    Suppose that \(G\) is a finite group with a unique minimal normal subgroup \(N\) and that \(x\), \(y\) are non-trivial elements of \(G\). If  $N$ is non-abelian, then  \(\abs{x^G}\) and \(\abs{y^G}\) are not coprime. In particular, \(\Gamma(G)\) is complete. 
\end{corollary}
\begin{proof}
    Suppose for a contradiction that \(\abs{x^G}\) and \(\abs{y^G}\) are coprime. Write \(N = T^s = T_1 \times \cdots \times T_s\) for \(s \geq 1\) and \(T\) a non-abelian simple group. Let \(r \in \pi(N)\). By \cref{lem:fund lemma}, we may assume that \(C_G(x)\) contains \(R \in \Syl_r(N)\). Since each \(T_i\) intersects \(R\) non-trivially we see that \(x\) must normalise each \(T_i\). Assume that \(y\) does not centralise any non-trivial Sylow subgroup of \(N\). Then for any prime \(t \in \pi(N)\) we have that \(C_G(x)\) contains a Sylow \(t\)-subgroup of \(N\). In particular, \(x \in C_G(N) = 1\), a contradiction. Therefore for some \(t \in \pi(N)\) we have that \(y\) centralises a Sylow \(t\)-subgroup of \(N\) and so \(y\) normalises every \(T_i\).

    Let \(M = \bigcap_{i=1}^s N_G(T_i) \norm G\). Then \(x\), \(y \in M\) and \(\abs{x^M}\) and \(\abs{y^M}\) are coprime by \cref{CollectionLemma}. By replacing $x$ and $y$ by suitable $G$-conjugates if necessary, neither $x$ nor $y$ centralises $T_1$. Let \(\theta\) be the projection from \(M\) to \(\Aut (T_1)\). Then $\abs{\theta(x)^{\theta(M)}}$ and $\abs{\theta(y)^{\theta(M)}}$ are coprime by \cref{CollectionLemma}. This contradicts \cref{AlmostSimple}. 
\end{proof}

With this corollary in mind, we can now prove \cref{ThmBcor2}.

\begingroup
\def\theintrothm{\ref{ThmBcor2}}
\begin{introthm}
    Suppose that $G$ is a finite group and $x$, $y \in G$ are in conjugacy classes of coprime sizes. Then the normal subgroup $\gen{x^G}\cap \gen{y^G}$ is abelian. Furthermore, we can write 
    $$\gen{x^G}\cap \gen{y^G}= (\gen{x^G}\cap Z(\gen{y^G}))(Z(\gen{x^G}) \cap \gen{y^G}).$$
\end{introthm}
\addtocounter{introthm}{-1}
\endgroup

\begin{proof} 
    Set $K = \gen{x^G}$ and $L = \gen{y^G}$. We first show that $T = K \cap L$ is soluble. To do this, by \cref{CollectionLemma}, we may assume for the remainder of this paragraph that the largest soluble normal subgroup of $G$ is trivial and show that $T = 1$.  Let $N$ be a minimal normal subgroup of $G$.  Then $N$ is non-abelian as $G$ has no soluble normal subgroups. Hence $NC_G(N)/C_G(N)$ is non-abelian and is the unique minimal normal subgroup of $G/C_G(N)$. Applying \cref{thmBcor} yields either $K \le C_G(N)$ or $L \le C_G(N)$. Hence $T \le C_G(N)$ for all minimal normal subgroups $N$ of $G$. Since $F(G) = 1$, $F^*(G)$ is the product of the minimal normal subgroups of $G$ and so $T \le C_G(F^*(G)) = Z(F^*(G)) = Z(F(G)) = 1$.
    
    We now show that $T$ is abelian. Since $T$ is soluble, $F^*(T)= F(T)$. Let $r\in \pi(F(T))$. Then, by \cref{lem:fund lemma} either $x$ or $y$ centralises $O_r(T)$. Since $O_r(T)$ is normal in $G$, we have that $O_r(T)$   is centralised by $K$ or by $L$ and so is centralised by $T$. Since $F(T)=\gen{O_r(T)\mid r\in \pi(t)}$, we deduce that $T\le C_T(F(T))=Z(F(T))$ is abelian, and, in fact $T= (T \cap Z(K))(T \cap Z(L))= (K \cap Z(L))(L\cap Z(K))$.
\end{proof}

\cref{ThmBcor2} has the following direct consequence. 

\begin{corollary}\label{cor:new} 
    Suppose that $G$ is a finite group and $x$, $y \in G$ are in conjugacy classes of coprime sizes. If $\gen{x^G} \ge \gen{y^G}$, then $\gen{y^G}$ is abelian.
\end{corollary}

Before we prove \cref{thmB}, we give an example which provides an instance of \cref{cor:new} yet shows that we cannot expect to show that \(\gen{x^G} \cap \gen{y^G}\) is central in both \(\gen{x^G}\) and \(\gen{y^G}\).
 
\begin{example}
    Suppose that $G= \mathrm{Alt}(4)$. Then the elements of order $2$ are in a conjugacy class of size $3$ and the elements of order $3$ are in  conjugacy classes of size $4$.
\end{example}

\begingroup
\def\theintrothm{\ref{thmB}}
\begin{introthm}    
    Suppose that \(G\) is a finite group and \(x\), \(y \in G\) lie in conjugacy classes of coprime sizes. Assume that $K$ is a component of $G$ or $K=O_r(G)$ for some prime $r$. Then \(\gen{K^G}\) is centralised by \(x\) or by \(y\). In particular, if \(G\) is almost simple then \(x = 1\) or \(y = 1\).
\end{introthm}
\addtocounter{introthm}{-1}
\endgroup

\begin{proof} If $K= O_r(G)$, the result follows from \cref{lem:fund lemma}.
    If \(K\) is a component of $G$, \cite[(31.4)]{aschbacherfg} yields that every normal subgroup of \(G\) either centralises \(K\) or contains \(\gen{K^G}\). In particular, if neither \(\gen{x^G}\) nor \(\gen{y^G}\) centralise \(K\), we must have that \(\gen{K^G} \leq \gen{x^G} \cap \gen{y^G}\). However, \(\gen{x^G} \cap \gen{y^G}\) is abelian by \cref{ThmBcor2}. This contradiction proves that at least one of \(\gen{x^G}\) or \(\gen{y^G}\) centralises \(\gen{K^G}\).
\end{proof}

Given \cref{thmB}, the following observation is immediate.

\begin{corollary}
    Let \(G\) be a finite group and suppose \(x\), \(y \in G\) lie in conjugacy classes of coprime sizes. If \(N\) is a minimal normal subgroup of \(G\) then at least one of \(x\) or \(y\) centralises \(N\).
\end{corollary}

We conclude this section by observing that \cref{thmA} follows from \cref{ThmBcor2}.

\begingroup
\def\theintrothm{\ref{thmA}}
\begin{introthm}   
    Let $G$ be a finite group and $\pi$ be a set of primes. If $x$ and $y$ are \(\pi\)-regular elements of $G$ whose conjugacy classes have coprime sizes, then $x^G y^G$ is a \(\pi\)-regular conjugacy class.
\end{introthm}
\addtocounter{introthm}{-1}
\endgroup

\begin{proof}
    Assume $x$, $y \in G$ have coprime class sizes and $x$ and $y$ are $\pi$-regular for a set of primes $\pi$. Then $[x,y] \in S$ where $S$ is the largest soluble normal subgroup of $G$ by \cref{ThmBcor2}. Hence $\gen{x,y}S/S$ is abelian and $\gen{x,y}S$ is soluble. Now \cref{lem:not sol sbgrp} shows that $x^Gy^G=(xy)^G$ is a $\pi$-regular conjugacy class. This proves \cref{thmA}.
\end{proof}

\section{Some consequences of \texorpdfstring{\cref{thmA}}{Theorem D}}  \label{sec:corollaries}

As mentioned in the Introduction, for $p$ a prime, several published results concerning graphs associated with $p$-regular conjugacy classes of a group have been proved assuming that the group is $p$-soluble. In many cases, this assumption was used only to ensure that the product of two $p$-regular conjugacy classes of coprime sizes is itself a $p$-regular conjugacy class. Now, as a consequence of \cref{thmA}, these results can be extended. In this final section, we collect the corresponding generalisations.

In \cite{LZ}, Lu and Zhang introduced the graph $\Delta_p(G)$ on the set of $p$-regular conjugacy class sizes of a finite group $G$. The vertices of $\Delta_p(G)$ are the prime numbers that divide the size of some $p$-regular conjugacy class, and an edge joins two vertices $r$ and $s$ whenever their product $rs$ divides the size of some $p$-regular conjugacy class. In \cite{bipartite}, Iranmanesh and Praeger observe that there is a close connection between $\Delta_p(G)$ and $\Gamma_p(G)$ which is illuminated by considering the bipartite divisor graph. This graph has vertices the disjoint union of the vertices of $\Delta_p(G)$ and $\Gamma_p(G)$, and a prime $r$ is joined to $|x^G|$ if $r$ divides $|x^G|$. Consideration of this graph makes it clear that \(\Delta_p(G)\) and \(\Gamma_p(G)\) have the same number of connected components and the diameters of these connected components differ by at most 1 (this has previously been noted by Lewis \cite[Corollary 3.2]{Lewis}).
 
This first corollary is an extension of the main results in \cite{BF1, BF2, LZ} and, in particular, generalises \cite[Theorem 4]{BF1} and \cite[Theorem A]{BF2} for $\Delta_p(G)$, and \cite[Theorems 1,2,3]{BF1} for  $\Gamma_p(G)$.

 \begin{corollary}\label{cor:1}
    Let $G$ be a finite group and $p$ be a prime. Then the graphs $\Delta_p(G)$ and $\Gamma_p(G)$ are either connected with diameter at most three or have exactly two connected components, both of which are complete.  
\end{corollary}

\begin{proof}
    It is sufficient to adapt the proof of \cite[Lemma 1]{LZ} and all the results in \cite{BF1, BF2}, taking into account \cref{thmA}.
\end{proof}

We also provide a generalisation of \cite[Theorem 2]{BF5} to use in the proofs of \cref{cor_5,cor_discon}.

\begin{corollary}\label{cor:new cor} 
    Suppose that $G$ is a finite group. Let $B$ be a $p$-regular conjugacy class of maximal size and write 
        $$M = \gen{D \mid D \text{ is a } p\text{-regular conjugacy class and } (\abs{D},\abs{B}) = 1}.$$
    Then $M$ is an abelian $p'$-subgroup of $G$. Furthermore, $Z_{p'}=Z(G) \cap O_{p'}(G) \le M$ and $\pi(M/Z_{p'}) \subseteq \pi (\abs{B})$.
\end{corollary}

\begin{proof} 
    This follows as the proof of \cite[Theorem 2]{BF5} since it only requires \cite[Lemma 1]{LZ} and this holds for an arbitrary finite group by \cref{thmA}.
\end{proof}

A group $G$ is called a \emph{quasi-Frobenius} group if $G/Z(G)$ is a Frobenius group. In this context we refer to the preimages in \(G\) of the Frobenius kernel and of Frobenius complements of \(G/Z(G)\) as the kernel and complements of \(G\), respectively.
 
The next result extends  \cite[Theorem 5]{BF1}. 
\begin{corollary}   \label{cor_5}
    Let $G$ be a finite group and $p$ be a prime. Suppose that $\Gamma_p(G)$ is not connected.
    \vspace*{-3mm}
    \begin{enumerate}[i)]
        \item If $p$ does not divide any vertex of $\Gamma_p(G)$, then $G = P \times H$ where $P \in \Syl_p(G)$ and $H$ is a $p$-complement of $G$ which is a quasi-Frobenius group with abelian kernel and complements.
        \item If $p$ divides some vertex belonging to the connected component which does not contain the largest size of a $p$-regular class, then $G$ is $p$-nilpotent and the normal $p$-complement of $G$ is a quasi-Frobenius group with abelian kernel and complements.
    \end{enumerate}
\end{corollary}

\begin{proof} 
    For part i) we apply a result of Camina \cite[Lemma 1]{C} to see that the Sylow $p$-subgroup $P$ is a direct factor of $G$. This replaces  \cite[Proposition 2]{BF1}. Then application of \cite[Proposition 3]{BF1} proves i). 

    Part ii) follows as in the proof of \cite[Theorem 5]{BF1} using $K_1=M$ of \cref{cor:new cor} and applying \cite[Proposition 3]{BF1}.
\end{proof}

\begin{remark}
    Suppose that $\Gamma_p(G)$ is not connected. Let us denote by $X_1$ and $X_2$ its two (complete) connected components and assume that $X_2$ contains a $p$-regular class of maximal size. In this situation, the previous result shows that when either $p$ does not divide any vertex of $\Gamma_p(G)$, or $p$ divides some vertex in $X_1$, then $G$ is soluble. 

    This leaves the possibility that \(p\) divides some vertex in \(X_2\). This case was analysed in \cite{BF5} where it was proved that if \(G\) is \(p\)-soluble then it is soluble. However, we cannot immediately employ \cref{thmA} to remove the \(p\)-solubility hypothesis in this situation (see \cite[Theorem 9, Corollary 10, Theorem 12]{BF5}) and so whether the solubility conclusion extends to an arbitrary finite group remains an open question.
\end{remark}

The methods and results developed for $\Gamma_p(G)$ in \cite{BF1} have been used by the same authors in \cite{BF4} to study the structure of a $p$-soluble group $G$ in a range of different arithmetic situations for the set of sizes of its $p$-regular conjugacy classes. We state next the corresponding generalisations, removing the $p$-solubility assumption of $G$ by virtue of \cref{thmA}. 

\begin{corollary}   \label{cor_discon}
    Let $G$ be a finite group and $p$ be a prime. Denote by $m$ and $n$ the two largest sizes of $p$-regular conjugacy classes of $G$ and assume $m>n>1$. Suppose further that $(m,n)=1$ and $p$ does not divide $n$. Then $G$ is soluble, and the following hold:
    \vspace*{-3mm}
    \begin{enumerate}[i)]
        \item The set of sizes of the $p$-regular conjugacy classes of $G$ is $\{1, n, m\}$.
        \item A $p$-complement of $G$ is a quasi-Frobenius group with abelian kernel and complements. Furthermore, its set of class sizes is $\{1, n, m_{p'}\}$.
    \end{enumerate}
\end{corollary}

\begin{proof}
    The proof runs as in \cite[Theorem A]{BF4} with the following modifications. First of all, 
       \cite[Lemmas 1 and 2]{BF4} are true generally by \cref{thmA}, and the same holds for \cite[Theorem 3]{BF4} which should be replaced by \cref{cor:new cor}. Now Steps 1 to 8 of \cite[Theorem A]{BF4} can be mimicked. Step 9 can be proved without taking a $p$-complement of $\ce{G}{a}$, since one can simply work with each Sylow $r$-subgroup of $\ce{G}{a}$ for each prime $r \neq p$. Finally, Step 10 can also be proved without the initial sentence where a $p$-complement of $G$ is taken.
\end{proof}

\begin{corollary}   \label{cor:2}
    Let $G$ be a finite group and $p$ be a prime. Suppose that the sizes of the $p$-regular conjugacy classes of $G$ are $\{1, m, n\}$, with $(m,n) = 1$. Then $G$ is soluble, and the $p$-complements of $G$ are quasi-Frobenius groups with abelian kernel and complements. Moreover, the set of class sizes of a $p$-complement of $G$ is $\{1, m_{p'}, n_{p'}\}$.
\end{corollary}

\begin{proof}
    We can argue as in \cite[Corollary B]{BF4}, but taking into account that \cite[Theorem 5]{BF1} can be replaced by \cref{cor_5}, and \cite[Theorem A]{BF4} by \cref{cor_discon}.
\end{proof}

\begin{corollary}   
    Let $G$ be a finite group and $p$ be a prime. If the set of sizes of non-central $p$-regular conjugacy classes of $G$ is of the form $\{n, n+1, \ldots, n+r\}$ for suitable integers $n$ and $r$, then \(G\) is soluble and one of the following holds:
    \vspace*{-3mm}
    \begin{enumerate}[i)]
        \item $r = 0$, $n = p^a$ for some positive integer $a$, and $G$ has abelian $p$-complements.
        \item $r = 0$, $n = p^aq^b$ where $q \neq p$ is a prime, $a \geq 0$, $b \geq 1$, and $G=PQ \times A$ with $P \in \Syl_p(G)$, $Q \in \Syl_q(G)$ and $A \le Z(G)$. Moreover, if $a = 0$ then $G=P \times Q \times A$.
        \item $r = 1$ and any $p$-complement of $G$ is a quasi-Frobenius group with abelian kernel and complements. Further, if $p$ does not divide $n$ then $G$ is $p$-nilpotent and if, in addition, $p$ does not divide $n+1$ then $G = P \times H$ where $H$ is the $p$-complement of $G$. 
    \end{enumerate}
\end{corollary}

\begin{proof}
    We follow the proof of \cite[Theorem C]{BF4}, but replacing \cite[Theorem 6]{BF4} by \cite[Theorem A]{BF3}, \cite[Corollary B]{BF4} by \cref{cor:2}, \cite[Lemma 7]{BF4} by \cite[Lemma 1]{C}, and again \cite[Theorem 5]{BF1} by \cref{cor_5}.
\end{proof}

We also include the following statement, which extends \cite[Theorem D]{BF4} to the general case. However, note that an elementary induction argument and Burnside's theorem yield that a group in which every $p$-regular conjugacy class of $G$ has prime power size is soluble. Thus our contribution is just to record the result for an arbitrary finite group. 

\begin{corollary}
    Let $G$ be a finite group and $p$ a prime. Every $p$-regular conjugacy class of $G$ has prime power size if and only if \(G\) is soluble and one of the following holds:
    \vspace*{-3mm}
    \begin{enumerate}[i)]
        \item $G$ has abelian $p$-complements. This occurs if and only if the size of every
        $p$-regular conjugacy class is a power of $p$.
        \item $G$ is nilpotent with abelian Sylow $r$-subgroups for all primes $r$ distinct from $p$ and from some prime $q \neq p$. This occurs if and only if the size of every $p$-regular class is a power of~$q$.
        \item $G = P \times H$, where $P \in \Syl_p(G)$ and $H$ is a $p$-complement
        of $G$. Furthermore, $H$ is quasi-Frobenius with abelian kernel and complements, and the set of conjugacy class sizes of $H$ is $\{1, q^s, r^t\}$ for positive integers $s$, $t$ and for some primes $q \neq r$, both distinct from $p$. This occurs if and only if the set of sizes of the $p$-regular conjugacy classes is exactly $\{1, q^s, r^t\}$.
    \end{enumerate}
\end{corollary} \qed

To close, as mentioned in \cite[Remark 2.2]{S}, the main result of that paper holds for an arbitrary group $G$ as a consequence of \cref{thmA}. 

\begin{corollary}
    Let $G$ be a finite group and $p$ be a prime. Then \(\Gamma_p(G)\) is a $k$-regular graph for some integer $k \geq 1$ if and only if it is a complete graph with $k+1$ vertices.
\end{corollary} 

As indicated in the Introduction, the above results can be regarded as generalisations of the corresponding results for \(\Gamma(G)\); in fact, it is enough to apply them with a prime $p$ not dividing the order of $G$ to recover the main results of \cite{A,BHM,BCHP,BCG,CHM}.

\printbibliography

\end{document}